\documentclass[11pt]{article}
\usepackage[pagebackref,colorlinks=true,pdfpagemode=none,urlcolor=blue,linkcolor=blue,citecolor=blue]{hyperref}

\usepackage{amsthm, amssymb, bm,array}
\usepackage{soul, color}
\usepackage[]{amsmath}
\usepackage[]{amsfonts}
\usepackage[]{fancyhdr}
\usepackage[]{graphicx}
\graphicspath{{/EPSF/}{../figures/}{figures/}}

\newtheorem{theorem}{Theorem}[section]
\newtheorem{lemma}[theorem]{Lemma}

\newtheorem{remark}[theorem]{Remark}
\newtheorem{hypothesis}[theorem]{Hypothesis}

\def \Imm {\mathbb{I}}

\def \Rm {\mathbb{R}}

\def\C{\mathcal{C}}

\def\F{\mathcal{F}}

\def\L{\mathcal{L}}

\def\N{\mathcal{N}}

\newcommand{\bc}{\mathbf c}

\newcommand{\be}{\mathbf e}
\newcommand{\bg}{\mathbf g}

\newcommand{\bu}{\mathbf u}

\newcommand{\tdiv}{\text{div }}

\newcommand{\where}{\quad\text{ where }}

\newcommand{\bfe}{ {\bf e}}
\newcommand{\bff}{ {\bf f}}
\newcommand{\bfg}{ {\bf g}}

\newcommand{\bfm}{ {\bf m}}

\newcommand{\bfu}{ {\bf u}}

\newcommand{\tr}{ {\text{tr }}}

\newcommand{\cout}[1]{}

\newcommand{\dprod}[2]{#1 : #2}

\def \eps {\epsilon}
\def \x {\mathrm{x}}

\def \d {D_V}
\newcommand{\wtC}{\widetilde{C}}

\title{Reconstruction of a fully anisotropic elasticity tensor from knowledge of displacement fields}
\author{Guillaume Bal\thanks{Department of Applied Physics and Applied Mathematics, Columbia University,  New York NY, 10027; gb2030@columbia.edu} \and Fran\c cois Monard\thanks{Department of Mathematics, University of Washington; fmonard@math.washington.edu} \and Gunther Uhlmann\thanks{Department of Mathematics, University of Washington; Hong Kong University of Science and Technology; gunther@math.washington.edu}}

\begin{document}

\maketitle

\begin{abstract}
    We present explicit reconstruction algorithms for fully anisotropic unknown elasticity tensors from knowledge of a finite number of internal displacement fields, with applications to transient elastography. Under certain rank-maximality assumptions satified by the strain fields, explicit algebraic reconstruction formulas are provided. A discussion ensues on how to fulfill these assumptions, describing the range of validity of the approach. We also show how the general method can be applied to more specific cases such as the transversely isotropic one.  
\end{abstract}

\section{Introduction} 

We consider the reconstruction of a fully anisotropic elasticity tensor $C = \{C_{ijkl}\}_{1\le i,j,k,l\le 3}$ from knowledge of a finite number of displacement fields $\{\bu^{(j)}\}_{j\in J}$, solutions of the system of linear elasticity
\begin{align}
\nabla\cdot (C: (\nabla \bu+(\nabla\bu)^T)) = 0 \qquad (X), \qquad \bu|_{\partial X} = \bg \quad (\text{prescribed}).
    \label{eq:elasticity}
\end{align}

Applications for such a theory include the medical imaging modality called elastography. Elastography is concerned with the reconstruction of the elastic properties in biological tissues. It has been observed experimentally that certain biological tissues (e.g., muscle fiber \cite{Gennisson2010}, or white matter inside the brain \cite{Feng2013}) display anisotropic mechanical properties, and the present article aims at giving access to these anisotropic features. The present approach to elastography consists of two steps. A first step is the reconstruction of the internal elastic displacements, which are accurately modeled as solutions of the above system of linear elasticity. Two different methods are typically used to reconstruct such displacements. One such method, called Ultrasound Elastography, consists of probing the elastic displacements by sound waves. We refer the reader to \cite{GCCF-JASA-03,OCPYL-UI-91,Oetal-JMU-02,STCF-IEEE-02} for more details on the method. A recent analysis of such reconstructions from ultrasound measurements is performed in \cite{BI-SIIMS-14}. A second method, called Magnetic Resonance Elastrography, leverages the displacement of protons by propagating elastic signals to reconstruct the elastic waves by magnetic resonance imaging \cite{JFK-SIAP-11,Kal-PMB-00,MLRGE-S-95}. 

In this paper, we assume the first step done and consider the quantitative reconstruction of elastic coefficients from knowledge of such elastic displacements. For references on this quantitative step of elastography, we refer the reader to \cite{Barbone2010,Bal2013b,MZM-IP-10}. 
Here, we generalize earlier work performed for scalar second-order equations \cite{Bal2012,Bal2012a} to the case of systems, and we generalize recent work on the Lam\'e system \cite{Bal2013b,Lai2013} to the fully anisotropic setting. 
%
%
Other hybrid inversions for elliptic systems can be found in the case of Maxell's system in \cite{Bal2013d,Bal2014a} and in a more general framework in \cite{B-CM-14}.

The approach consists in deriving explicit algebraic formulas reconstructing the unknown parameters locally, based on hypotheses of linear independence (or rank maximality) of functionals of the measurements and their partial derivatives. A discussion then follows on what regularity or property (e.g., the Runge approximation property) is required {\it a priori} on the unknown parameters so that the hypotheses of reconstructibility may be fulfilled. The idea of constructing local solutions fulfilling certain maximality conditions, which are then controlled from the boundary of the domain via Runge approximation, was also used in the context of reconstruction of conductivity tensors from knowledge of so-called {\em power density} functionals \cite{Monard2012a} or {\em current density} functionals \cite{Bal2013}.

The above problem of reconstruction of coefficients in partial differential equations from knowledge of internal functionals is sometimes referred to as a hybrid inverse problems; see \cite{A-Sp-08,AGKL-QAM-08,AS-IP-12,B-CM-14,KS-IP-12,S-SP-2011} for reference for hybrid inverse problems in other imaging modalities. Note that we do not consider here the inverse boundary elasticity problem, for which Lipschitz stability estimates may be obtained only when the Lam\'e coefficients are piece-wise constant; see for instance \cite{BFV-IPI-14} for a recent reference on such a topic not covered here.

We now give the main results of the paper in the next section and give an outline of the remainder of the paper there.  

\section{Main results}\label{sec:main} 

\paragraph{Preliminary notation and definitions.} In what follows, we denote by $M_3(\Rm)$ the vector space of $3\times 3$ real matrices with inner product $A:B := \tr (AB^T) = \sum_{i,j=1}^3 A_{ij}B_{ij}$, with respect to which we recall the orthogonal decomposition $M_3(\Rm) = S_3(\Rm) \oplus A_3(\Rm)$, where the first (second) summand denotes (skew-)symmetric matrices. 

Let us fix $X\subset \Rm^3$ a bounded domain with smooth boundary for the remainder of the paper. An {\em elasticity tensor} $C$ is a fourth-order tensor satisfying the following symmetries
\begin{align*}
    C_{ijkl} = C_{jikl}, \qquad C_{ijkl} = C_{ijlk}, \qquad C_{ijkl} = C_{klij}, \qquad 1\le i,j,k,l\le 3,
\end{align*}
characterized by 21 independent components (instead of 81), where the latter symmetry corresponds to the assumption that $C$ is {\em hyperelastic}. We assume below that $C$ is {\em uniformly pointwise stable} over $X$ \cite[Ch. 6, Def. 1.9]{Marsden1983} in the sense that there is a $\kappa>0$ such that 
\begin{align}
    \frac{1}{2} \eps:C(\x):\eps \ge \kappa\ \eps:\eps, \quad \forall\ \x\in X, \quad \forall\ \eps\in S_3(\Rm).
    \label{eq:UPS}
\end{align}
For $C = \{C_{ijkl}\}_{i,j,k,l}$ an elasticity tensor and $\eps\in S_3(\Rm)$, we denote $C:\eps := \{\sum_{k,l} C_{ijkl} \eps_{kl}\}_{i,j}$. 

In some sections below, the tensor $C$ will be represented in the non-tensorial {\em Voigt} notation, an $S_6(\Rm)$-valued function $\bc = \{\bc_{\alpha\beta}\}_{1\le \alpha,\beta \le 6}$, where the correspondence of elements $c_{\alpha\beta}$ in terms of the coefficients $C_{ijkl}$ is obtained via the double index mapping $11\mapsto 1$, $22\mapsto 2$, $33\mapsto 3$, $23,32\mapsto 4$, $13,31\mapsto 5$ and $12,21\mapsto 6$ (for instance, $\bc_{11} = C_{1111}$ and so on). Hooke's law, relating stress $\sigma$ to strain tensors $\eps$ via the relation $\sigma = C:\eps$, now reads in Voigt notation $\sigma_V = \bc\ \eps_V$, where we define 
\begin{align}
    \eps_V = (\eps_{11}, \eps_{22}, \eps_{33}, 2\eps_{23}, 2\eps_{31}, 2\eps_{12})^T, \qquad \sigma_V = (\sigma_{11}, \sigma_{22}, \sigma_{33}, \sigma_{23}, \sigma_{31}, \sigma_{12})^T. 
    \label{eq:Voigt}
\end{align}
Most often, we will drop the subscript ``$V$'' below as the context will tell us naturally what representation to pick. For $(\eps^{(1)}, \dots, \eps^{(6)})$ in $S_3(\Rm)$, we define below
\begin{align}
    \det_V (\eps^{(1)}, \dots, \eps^{(6)}) = \det_{\Rm^6} (\eps_V^{(1)}, \dots, \eps_V^{(6)}). 
    \label{eq:detv}
\end{align}

A crucial fact for further derivations is the following.
\begin{lemma} \label{lem:detc}
    Under assumption \eqref{eq:UPS}, there exists $\kappa'>0$ such that
    \begin{align}
	\det\bc(\x) \ge \kappa', \qquad \forall\ \x\in X.
	\label{eq:detc}
    \end{align}    
\end{lemma}

\paragraph{Main results.} Provided linear independence conditions that can be satisfied locally and checked directly on the available measurements, we first provide an explicit reconstruction algorithm for fully anisotropic elasticity tensors. Here and below, a typical displacement field is denoted $\bfu:X\to \Rm^3$ with corresponding strain tensor $\eps: X\to S_3(\Rm)$ with components $\eps_{ij} = \frac{1}{2} (\partial_i u_j + \partial_j u_i)$, or, in coordinate-free notation, $\eps = \frac{1}{2} (\nabla \bu + (\nabla\bu)^T)$. \\
We now formulate our main hypotheses in order to set up our reconstruction procedure. These hypotheses, based on algebraic redundancies of various elasticity solutions, force the unknown tensor to lie on the orthogonal of a space generated by a rich enough set of data. The hypotheses below formulate how some functionals of the data set can be made to generate a hyperplane of $S_6(\Rm)$, a normal of which can be explicitely constructed and proved to be proportional to $C$, the proportionality factor being reconstructed at the end. As seen below, fulfilling this agenda requires $6+N$ solutions, where $N \ge \frac{d-1}{3}$ and $d$ denotes the number of scalar components of $C$ (the factor $3$ accounts for the fact that each additional solution after the sixth one provides three orthogonality constraints on $C$).

\begin{hypothesis}\label{hyp} Let $\Omega\subset X$. 
    \begin{itemize}
	\item[A.] There exist 6 solutions $\bfu^{(1)}, \dots, \bfu^{(6)}$ of \eqref{eq:elasticity} whose strain tensors form a basis of $S_3(\Rm)$ at every $\x\in \Omega$. This condition can be summarized as ($\det_V$ is defined in \eqref{eq:detv}) 
	    \begin{align}
		\inf_{\x\in\Omega} \det_V (\eps^{(1)}(\x), \dots, \eps^{(6)}(\x)) \ge c_0 >0, \quad \text{for some constant } c_0.
		\label{eq:hyp1}
	    \end{align}
	\item[B.] Assuming $A$ fulfilled, there exists $N$ additional solutions $\bfu^{6+1}, \dots, \bfu^{6+N}$ giving rise to a family $M$ of $3N$ matrices whose expressions are explicit in terms of $\{\eps^{(j)}, \partial_\alpha \eps^{(j)},\ 1\le\alpha\le 3,\ 1\le j\le 6+N\}$ (see \eqref{eq:orthogonality} for detail), and such that they span a hyperplane of $S_6(\Rm)$ at every $\x\in \Omega$. This condition can be summarized as 
	    \begin{align}
		\inf_{\x\in\Omega} \sum_{M'\subset M,\ \# M' = 20} \N(M') :\N(M') \ge c_1 >0, \quad \text{for some constant } c_1, 
		\label{eq:hyp2}
	    \end{align}
	    where $\N$ is an operator generalizing the cross-product, defined in \eqref{eq:crossprod}.
    \end{itemize}
\end{hypothesis}

\begin{remark}
    It should be noted that these hypotheses are stable under smooth perturbation of the boundary conditions generating the displacement fields $\bfu^{(1)}, \dots, \bfu^{(6+N)}$. This is because Hypotheses \ref{hyp}.A-B are expressed in terms of functionals which depend polynomially on the components of displacement fields and their derivatives up to second order, which are in turn continuous functionals of their boundary conditions (see, e.g., Sec. \ref{sec:prelim} for appropriate topologies). 
\end{remark}

\begin{remark}
    In Hypothesis \ref{hyp}.B, the number $N$ depends on the type of isotropy of the tensor $C$. In the most general, $21$-parameter case, spanning a hyperplane of $S_6(\Rm)$ with $3N$ additional constraints suggests $N\ge 7$, and the total number of displacement fields needed is then $6+7=13$. 
\end{remark}

On an open subset $\Omega\subset X$ where Hypotheses \ref{hyp}.A-B hold, we then derive an explicit reconstruction algorithm, reconstructing $C$ over $\Omega$. This is done as follows: in Voigt notation, decompose the $S_6(\Rm)$-valued function $C$ as the product of a scalar function $\tau$ times a normalized anisotropic structure $\wtC$ such that $C = \tau\wtC$ ($\wtC$ is such that its Voigt counterpart $\tilde\bc$ has determinant $1$, so $\tau = (\det\tilde\bc)^{\frac{1}{6}}$). Algebraic manipulations allow us to obtain pointwise orthogonality constraints which, under hypotheses \eqref{eq:hyp1}-\eqref{eq:hyp2}, are numerous enough to reconstruct $\wtC$ pointwise. Once $\wtC$ is reconstructed, an equation for $\nabla\log\tau$ is derived over $\Omega$, leading to the reconstruction of $\tau$ by solving either a transport or a Poisson equation. Finally, once $C=\tau\wtC$ is known, additional stability is recovered on certain components of $C$ by deriving equations reconstructing the third-order tensor $\tdiv C$ (whose components may be written as $\partial_i C_{ijkl},\; 1\le j,k,l \le 3$). Such components, although linear combination of first-order derivatives of the components of $C$, are reconstructed with the same stability as $C$ itself. Such stability improvements have already been observed in e.g. \cite{Bal2012e,Bal2013}.

The approach just described yields unique and stable reconstructions over $\Omega$ in the sense of the following theorem. 

\begin{theorem} \label{thm:injstab}
    Suppose that over some open set $\Omega\subset X$, hypotheses \ref{hyp}.A-B hold for two families of displacement fields $\{\bfu^{(j)}\}_{j=1}^{6+N}$ and $\{\bfu^{'(j)}\}_{j=1}^{6+N}$ corresponding to elasticity tensors $C$ and $C'$. Then $C$ and $C'$ can each be uniquely reconstructed over $\Omega$ from knowledge of their corresponding solutions, with the following stability estimate for every integer $p\ge 0$
    \begin{align}
	\|C - C'\|_{W^{p,\infty} (\Omega)} + \|\tdiv C - \tdiv C'\|_{W^{p,\infty}(\Omega)} \le K \sum_{j=1}^{N+6} \|\eps^{(j)} - \eps^{'(j)}\|_{W^{p+1,\infty}(\Omega)}.
	\label{eq:stab}
    \end{align}
\end{theorem}

\begin{remark}\label{rem:stabtau} Note that if the elasticity tensor is split into the product of an unknown scalar function $\tau$ times a {\em known} anisotropic tensor, the stability of the problem of reconstructing $\tau$ from fields $\bu^{(j)}$ is better-posed, i.e. involves the loss of one derivative instead of two, according to the statement
    \begin{align}
    \begin{split}
	\|\tau - \tau'\|_{W^{p+1,\infty}(\Omega)} \le K \sum_{j=1}^{N+6} \|\eps^{(j)} - \eps^{'(j)}\|_{W^{p+1,\infty}(\Omega)}.
    \end{split}
    \label{eq:stabtau}
\end{align}   
\end{remark}

In the context of this reconstruction approach, an elasticity tensor is then reconstructible on some given open set $\Omega$ (or globally on $X$) if there exist displacement field solutions of \eqref{eq:elasticity} fulfilling Hypotheses \ref{hyp}.A-B throughout $\Omega$. In this context, classifying reconstructible elasticity tensors then consists in finding out in what situations one can construct such displacement fields. Two such situations are described in the following theorem: (i) a fully anisotropic tensor that is close to constant is globally reconstructible from well-chosen displacement fields whose traces are explicitly given, and (ii) an elasticity tensor with smooth enough components which satisfies the {\em Runge approximation property} (see Section \ref{sec:runge}, in particular, Eq. \eqref{eq:L2est}) is, in principle, locally reconstructible from knowledge of its displacement fields.

\begin{theorem}[Reconstructibility of elasticity tensors]  \label{thm:recons} In either of the following cases, there exists a non-empty open set of smooth enough boundary conditions generating displacement fields characterizing $C$ uniquely and stably in the sense of Theorem \ref{thm:injstab}.  
    \begin{itemize}
	\item[(i)] $C$ is $\C^3$-close to constant. 
	\item[(ii)] $C$ is smooth (at least of class $\C^3$) and satisfies the Runge approximation property. 
    \end{itemize}
\end{theorem}

An application of particular interest is the reconstruction of transversely isotropic (TI) elasticity tensors. Such tensors have, as of yet, the highest type of anisotropy for which the Runge approximation has been proved (see \cite{Nakamura2005}), so that they are covered in case (ii) of Theorem \ref{thm:recons}. We also explain in Section \ref{ssec:constant} how to generate explicit boundary conditions reconstructing a constant TI tensor which, by case (i) of Theorem \ref{thm:recons}, can be used to reconstruct a near-constant TI tensor.

\paragraph{Outline.} The rest of the paper is organized as follows. Section \ref{sec:prooflemma} covers the proof of Lemma \ref{lem:detc}. Section \ref{sec:algo} covers the reconstruction procedure as well as its stability property (proof of Theorem \ref{thm:injstab}). Section \ref{sec:recons} covers the two ways described in Theorem \ref{thm:recons} to fulfill Hypotheses \ref{hyp}.A-B for certain classes of tensors, thereby establishing their unique and stable (in the sense of Theorem \ref{thm:injstab}) reconstructibility from displacement fields.

\section{Proof of Lemma \ref{lem:detc}} \label{sec:prooflemma}

Suppose the elasticity tensor $C$ satisfies \eqref{eq:UPS}. In the study of elastic eigentensors in \cite{Mehrabadi1990}, it is established that $C$ can have at most six distinct eigencouples $(\eps,\lambda)$ such that the relation $C:\eps = \lambda\eps$ is satisfied. The symmetries of $C$ and hypothesis \eqref{eq:UPS} imply that the $\lambda$'s are real and that all of them satisfy $\lambda\ge \kappa$. It is also mentioned in \cite{Mehrabadi1990} that the eigenvalues $\lambda$ are precisely the eigenvalues of the $S_6(\Rm)$-valued representation of $C$ in the form 
\begin{align*}
    \bc' = \{ 2^{\frac{\chi(i) + \chi(j)}{2}} c_{ij} \}_{1\le i,j \le 6}, \quad \text{where} \quad \chi(i) = \left\{
    \begin{array}{cc}
	0 & \text{if }\; i = 1, 2, 3, \\
	1 & \text{if }\; i = 4, 5, 6,
    \end{array}
    \right.
\end{align*}
and where $c_{ij}$ denotes the Voigt representation of $C$ presented in the introduction. It is then immediate that $\det \bc' = 8 \det \bc$. Moreover, as mentioned above, all eigenvalues of $\bc'$ match the eigenvalues of $C$ and therefore satisfy the estimate $\lambda\ge \kappa$. Therefore
\[ \det\bc = \frac{1}{8} \det \bc' \ge \frac{\kappa^6}{8}, \]
hence \eqref{eq:detc} holds with $\kappa' = \frac{\kappa^6}{8}$.

\section{Reconstruction algorithm and its stability} \label{sec:algo}

As seen in Theorem \ref{thm:injstab}, the left hand side of \eqref{eq:stab} contains two terms whose product forms the elasticity tensor $C$, reconstructed with stability estimates in different norms: we decompose $C$ into the product $\tau\wtC$, where $\wtC$ contains the rescaled anisotropic structure of $C$, defined by a normalizing condition (specifically, if $\tilde\bc$ describes $\wtC$ in Voigt notation, then $\det \tilde\bc = 1$) and $\tau$ is the remaining scalar factor. That this is possible comes from estimate \eqref{eq:detc}, which states that $\det\bc$ is uniformly bounded away from zero throughout $X$. The next two sections focus on the successive reconstruction of $\wtC$, then $\tau$ upon assuming that Hypotheses \ref{hyp}.A-B holds.

\subsection{Reconstruction of the anisotropy}

In this section, we will use the Voigt notation, 
so that strain tensors will be represented as $\Rm^6$-valued funtions $\eps \equiv \eps_V$.
In this notation, the elasticity tensor becomes an $S_6(\Rm)$-valued function, characterized by 21 scalar functions $\bc = \{\bc_{\alpha\beta}\}_{1\le \alpha,\beta\le 6}$ satisfying $\bc_{\alpha\beta} = \bc_{\beta\alpha}$ and where Hooke's law is expressed as a regular matrix-vector product $\sigma_V = \bc\ \eps_V$, with $(\sigma_V, \eps_V)$ as in \eqref{eq:Voigt}. This makes the problem tractable via algebraic manipulations on matrices instead of 4-tensors.

Using the Voigt notation, the elasticity system \eqref{eq:elasticity} takes the form 
\begin{align}
  \d\cdot (\bc\ \eps) = 0, \qquad \d := \left(
  \begin{array}{cccccc}
    \partial_1 & 0 & 0 & 0 & \partial_3 & \partial_2 \\
    0 & \partial_2 & 0 & \partial_3 & 0 & \partial_1 \\
    0 & 0 & \partial_3 & \partial_2 & \partial_1 & 0 
  \end{array}
  \right),
  \label{eq:elasticity2}
\end{align}
and where, for a $M_6(\Rm)$-valued function $A$ and a scalar function $f$, we have the identity 
\begin{align}
  \d \cdot (f A) = (\d f)\cdot A + f \d \cdot A.
  \label{eq:identity}
\end{align}

As the reconstruction approach is local (even pointwise for the anisotropic part), we assume to have 6 elasticity solutions $\{\bu^{(j)}\}_{1\le j\le 6}$ whose corresponding strain tensors $\{\eps_V^{(j)}(\x)\}_{1\le j\le 6}$ form a basis of $\Rm^6$ for every $\x$ of some subdomain $\Omega\subset X$. Any additional solution $\bu^{(p)}$ with $p\ge 7$ is such that for $\x\in\Omega$, $\eps^{(p)}(\x)$ decomposes uniquely into the basis above as 
\begin{align*}
    \eps^{(p)}(\x) = \sum_{j=1}^6 \mu_{pj}(\x) \eps^{(j)}(\x), \where\quad \mu_{pj} := \frac{\det_V (\eps^{(1)}(\x),\dots,\overbrace{\eps^{(p)}(\x)}^j, \dots, \eps^{(6)}(\x))}{\det_V (\eps^{(1)}(\x), \dots, \eps^{(6)}(\x))}, \quad 1\le j\le 6.
\end{align*}
Plugging this equality into the elasticity equation, we obtain that 
\begin{align*}
  0 = \d \cdot (\bc \eps^{(p)}) = \sum_{j=1}^6 \d\cdot (\mu_{pj} \bc\eps^{(j)}(\x_0)) = \sum_{j=1}^6 (\d \mu_{pj}) \cdot \bc \eps^{(j)} + \mu_{pj} \d\cdot (\bc \eps^{(j)}),
\end{align*}
which, since $\d\cdot (\bc \eps^{(j)}) = 0$, implies that 
\begin{align*}
  \sum_{j=1}^6 (\d \mu_{pj}) \cdot \bc \eps^{(j)} = 0.
\end{align*}
This last equation can be seen as three scalar orthogonality constraints on the tensor $\bc$ in the inner product structure of $S_6(\Rm)$ which we denote by $A:B := \tr (AB^T) = \sum_{i,j=1}^6 A_{ij} B_{ij}$, where the matrices that are othogonal to $C$ are directly known from available measurements. The last equation is equivalent to 
\begin{align}
  \begin{split}
    \bc: M^{(p),1} = \bc: M^{(p),2} &= \bc:M^{(p),3} = 0, \where \\
    M^{(p),1} &= \sum_{j=1}^6 (\partial_1\mu_{pj}\ 0\ 0\ 0\ \partial_3 \mu_{pj}\ \partial_2 \mu_{pj}) \otimes \eps^{(j)}, \\
    M^{(p),2} &= \sum_{j=1}^6 (0\ \partial_2\mu_{pj}\ 0\ \partial_3 \mu_{pj}\ 0\ \partial_1 \mu_{pj}) \otimes \eps^{(j)}, \\
    M^{(p),3} &= \sum_{j=1}^6 (0\ 0\ \partial_3\mu_{pj}\ \partial_2 \mu_{pj}\ \partial_1 \mu_{pj}\ 0) \otimes \eps^{(j)}.
  \end{split}
  \label{eq:orthogonality}
\end{align}
Note that since $\bc$ is orthogonal to $A_6(\Rm)$, one could replace the matrices $M^{(p),i}$ with their symmetrized versions. Notice that the components of these matrices are first partial derivatives of rational functions of strain tensors.  

If enough linear constraints of the form \eqref{eq:orthogonality} are available from a rich enough set of measurements, that is to say, if enough such matrices are available and form a hyperplane of $S_6(\Rm)$ at $\x_0$, then the tensor $\bc(\x_0)$, constrained to be perpendicular to this hyperplane, will be determined up to a multiplicative constant. The reconstruction can be done via a generalization of the cross-product, as used for instance in \cite{Monard2012a} for similar purposes in the context of the conductivity equation. Define $\{\bfm_j\}_{j=1}^{21}$ a basis of $S_6(\Rm)$ and given $M = \{M_j\}_{j=1}^{20}\subset S_6(\Rm)$, define the operator $\N: S_6(\Rm)^{20}\to S_6(\Rm)$ by expanding the formal determinant below with respect to its last row  
\begin{align}
    \N (M) := \frac{1}{\det (\bfm_1,\cdots,\bfm_{21})} \left|
    \begin{array}{ccc}
	\dprod{M_1}{\bfm_1} & \cdots & \dprod{M_1}{\bfm_{21}} \\
	\vdots & \ddots & \vdots \\
	\dprod{M_{20}}{\bfm_1} & \cdots & \dprod{M_{20}}{\bfm_{21}} \\
	\bfm_1 & \cdots & \bfm_{21}
    \end{array}
    \right|.
    \label{eq:crossprod}
\end{align} 
$\N$ is a 20-linear, alternating map that does not depend on the choice of basis for $S_6(\Rm)$. $\N (M)$ is a vector that is normal to the hyperplane spanned by $M$ when $M$ is linearly independent, zero otherwise. In particular, if $M$ is a family of matrices known to be orthogonal to a given matrix $\bfm'$, then $\N(M)$ is either zero (if $\dim\text{span }M<20$) or proportional to $\bfm'$ (if $\dim\text{span }M = 20$). 
\smallskip

In light of this last comment, and assuming that a rich enough set of solutions of the elasticity system \eqref{eq:elasticity} gives rise to a family of matrices $M$ of the form \eqref{eq:orthogonality}, with cardinality greater than 20 and spanning a hyperplane of $S_6(\Rm)$ at a given point $\x_0 \in X$, then for any given $20$-uple $M'\subset M$, $\N(M')$ is either zero or proportional to $\bc(\x_0)$. Normalizing $\tilde\bc(\x_0) = \det (\bc (\x_0))^{\frac{-1}{6}} \bc(\x_0)$ and enforcing the condition $\tilde\bc_{11}>0$ (this is because in the proof of Lemma \ref{lem:detc}, we notice that the top-left $3\times 3$ block of $\bc$ matches that of $\bc'$ which is a symmetric positive definite matrix, so one must have $\bc_{11}>0$, $\bc_{22}>0$ and $\bc_{33}>0$), then we have the equality
\begin{align*}
    (\pm)_{M'} \N(M') = (\det (\N(M')))^{\frac{1}{6}} \tilde\bc(\x_0),
\end{align*}
for every $20$-uple $M'\subset M$, where $(\pm)_{M'}$ is the sign of the top-left entry of $\N(M')$. This equation is either trivial when $M'$ is linearly dependent, or reconstructs $\tilde\bc(\x_0)$ otherwise. Condition \ref{eq:hyp2} ensures that at least one subfamily $M'$ is linearly independent, so we sum the last equation over all subfamilies and arrive at the formula 
\begin{align}
    \tilde\bc (\x_0) = \left(\sum_{M'\subset M,\ \# M' = 20} (\det (\N(M')))^{\frac{1}{6}}\right)^{-1} \sum_{M'\subset M,\ \# M' = 20} (\pm)_{M'} \N(M').
    \label{eq:Crc}
\end{align}

\begin{remark} The expliciteness of formula \eqref{eq:Crc} is interesting in its own right and makes the stability of the problem straighforward to assess. On the other hand, the computation of $20\times 20$ determinants can be expensive, and methods constructing a normal to a hyperplane without implementing \eqref{eq:Crc} might reveal more practical. For a $20$-uple $M$, an example of a potentially faster method for finding a scalar multiple of $\N(M)$ is to form a $20\times 21$ matrix whose rows are the elements of $M$, and to find a vector in the nullspace of that matrix via Gaussian elimination. 
\end{remark}

Now that the anisotropic structure $\wtC$ (or equivalently, $\tilde\bc$) is reconstructed, we now explain how the multiplicative scalar $\tau$ can be reconstructed via a standard transport equation. 

\subsection{Reconstruction of the scalar factor $\tau$} \label{sec:rectau}

We now switch back to 4-tensor notation. Plugging the decomposition $C = \tau\wtC$ into the elasticity equation, where $\wtC$ is assumed to be known from the previous step, we obtain, for each elasticity solution considered,
\begin{align}
    (\wtC:\eps)\nabla\log\tau = - \tdiv (\wtC:\eps),     
    \label{eq:gradtau1}
\end{align}
where $(\wtC:\eps) \in S_3(\Rm)$ is not necessarily invertible. However, under the assumption that $\{\eps^{(j)}\}_{j=1}^6$ is a basis of $S_3(\Rm)$, by virtue of \eqref{eq:detc}, so is $\{\wtC:\eps^{(j)}\}_{j=1}^6$. Therefore, the identity tensor $\Imm_3$ decomposes into this basis by means of some functions $\mu_j(\x)$ such that
\begin{align*}
    \Imm_3 = \sum_{j=1}^6 \mu_j(\x)\ \wtC:\eps^{(j)}.
\end{align*}
If we denote $D_{ij} := (\wtC:\eps^{(i)}):(\wtC:\eps^{(j)})$ for $1\le i,j\le 6$, the $S_6(\Rm)$-valued function $D = \{D_{ij}\}$ is known and invertible (as the Grammian matrix of the basis $\{\wtC:\eps^{(j)}\}_{j=1}^6$), and the entries of its inverse are denoted by $D^{ij}$. In this case, the functions $\mu_j$ take the explicit form 
\begin{align*}
    \mu_j(\x) = \sum_{i=1}^6 D^{ij}\ \Imm_3 : (\wtC:\eps^{(i)}) = \sum_{i=1}^6 D^{ij}\ \tr (\wtC:\eps^{(i)}).  
\end{align*}
Now taking a linear combination of \eqref{eq:gradtau1} weighted by the functions $\mu_j$, we obtain
\begin{align}
    \nabla\log\tau = -\sum_{j=1}^6 \mu_j(\x)\ \tdiv (\wtC:\eps^{(j)}).
    \label{eq:gradtau2}
\end{align}
The right hand side of this equation is completely known and $\tau$ can therefore be reconstructed via either $(i)$ integration of \eqref{eq:gradtau2}, or, after taking divergence, $(ii)$ solving a Poisson equation on the domain $\Omega$ with known Neumann boundary conditions. As discussed in \cite{Bal2011a}, the approach $(ii)$ may be more robust to noise than the former, as the integration of ordinary differential equations in the presence of noisy measurements may strongly depend on the choice of integration path. Moreover, taking divergence of \eqref{eq:gradtau2} has the advantage of naturally removing the curl part of noisy right-hand sides in \eqref{eq:gradtau2}.

\subsection{Reconstruction of the tensor $\tdiv C$}

Now that both $\wtC$ and $\tau$ are reconstructed, we explain how to gain stability on the reconstruction of the third-order tensor $\tdiv C = \partial_i C_{ijkl}\ \be_j\otimes \be_k\otimes \be_l$. Note that this tensor is symmetric in the pair of indices $(k,l)$. Now the elasticity equation can be rewritten as
\begin{align*}
    (\partial_i C_{ijkl} ) \eps_{kl} + C_{ijkl} \partial_i \eps_{kl} = 0, \qquad 1\le j\le 3,
\end{align*}
or in contracted notation,
\begin{align}
    (\tdiv C)_{j\cdot\cdot}:\eps = - C_{ijkl} \partial_i \eps_{kl}, \qquad 1\le j\le 3.
    \label{eq:divC}
\end{align}
For every $1\le j\le 3$, $(\tdiv C)_{j\cdot\cdot}$ can be seen as an $S_3(\Rm)$-valued function, and as such, using again the assumption that $\{\eps^{(j)}\}_{j=1}^6$ is a basis of $S_3(\Rm)$, upon defining $E_{ij} = \eps^{(i)}:\eps^{(j)}$ for $1\le i,j\le 6$, the $S_6(\Rm)$-valued function $E = \{E_{ij}\}$ is uniformly invertible over $\Omega$ and we denote by $E^{ij}$ the components of its inverse. With such notation, Cramer's rule in $S_3(\Rm)$ then reads
\begin{align*}
    (\tdiv C)_{j\cdot\cdot} = E^{pq} ((\tdiv C)_{j\cdot\cdot}:\eps^{(p)}) \eps^{(q)}, \qquad 1\le j\le 3,
\end{align*}
and combining this with \eqref{eq:divC}, we obtain a reconstruction formula for the tensor $\tdiv C$
\begin{align}
    (\tdiv C)_{j\cdot\cdot} = - E^{pq} (C_{ijkl}\ \partial_i \eps^{(p)}_{kl})\ \eps^{(q)}, \qquad 1\le j\le 3.
    \label{eq:divCrc}
\end{align}

\subsection{Proof of Theorem \ref{thm:injstab}} \label{sec:proofstab}

\paragraph{Uniqueness.} Equations \eqref{eq:Crc}, \eqref{eq:gradtau2} and \eqref{eq:divCrc} give us explicit algorithms to reconstruct $\tilde\bc$, $\tau$ and $\tdiv C$ explicitely, with the only indeterminacy that $\tau$ is defined up to a multiplicative constant. This is because multiplying the elasticity system by a constant changes neither the equation, nor its solutions. We can remove this indeterminacy by assuming that $\tau$ is known at a (boundary) point or by fixing its value at a point. The constructive nature of the approach gives uniqueness. 

\paragraph{Stability on $\wtC$.} The functional reconstructing $\tilde\bc$ is a rational function of strain tensors and their first partial derivatives. If two sets of displacement fields $\{\bu^{(j)}\}_{j=1}^{6+N}$ and $\{\bu^{'(j)}\}_{j=1}^{6+N}$ both satisfy hypotheses \ref{hyp}.A-B over some $\Omega\subset X$, the denominator of the right hand side of \eqref{eq:Crc} never vanishes, and its rational expression in terms of measurement components yields the local estimate
\begin{align}
    \|\wtC -\wtC'\|_{L^\infty (\Omega)} \le K \sum_{j=1}^{6+N} \|\eps^{(j)}-\eps^{'(j)}\|_{W^{1,\infty}(\Omega)},
    \label{eq:stab1}
\end{align}
for some constant $K>0$. More generally, if the strain tensors are in $W^{p+1,\infty}(\Omega)$ for some integer $p \ge 0$, the reconstruction formula \eqref{eq:Crc} can be differentiated $p$ times to yield stability estimates in smoother norms of the form
\begin{align}
    \|\wtC -\wtC'\|_{W^{p,\infty} (\Omega)} \le K \sum_{j=1}^{6+N} \|\eps^{(j)}-\eps^{'(j)}\|_{W^{p+1,\infty}(\Omega)}, \qquad p\ge 0,
    \label{eq:stab2}
\end{align}
where the constant $K$ grows polynomially in terms of the maximum $W^{p+1,\infty}(\Omega)$ norm of the strain tensors and the inverse of $\min (c_0,c_1,c'_0,c'_1)>0$, where $(c_0,c_1,c'_0,c'_1)$ are the constants defined in \eqref{eq:hyp1}-\eqref{eq:hyp2} corresponding to each system of measurements $\{\bu^{(j)}\}_{j=1}^{6+N}$ and $\{\bu^{'(j)}\}_{j=1}^{6+N}$.

\paragraph{Stability on the scalar factor $\tau$.} Equation \eqref{eq:gradtau2} takes the form $\nabla\log\tau = F (\wtC, \bu)$ with $F$ a rational function of the components of $\wtC$ and its first derivatives, and of strain tensors and their first partial derivatives, one deduces
\begin{align}
    \|\tau-\tau'\|_{W^{p+1,\infty} (\Omega)} \le K' \sum_{j=1}^{6+N} \|\eps^{(j)}-\eps^{'(j)}\|_{W^{p+1,\infty}(\Omega)} + K'' \|\wtC - \wtC'\|_{W^{p+1,\infty}(\Omega)}, \qquad p\ge 0.
    \label{eq:stab3}
\end{align}
We see that when $\wtC$ is known {\it a priori}, only one derivative is lost from the measurements to the quantity $\tau$. However, when considering joint reconstructions of $(\wtC,\tau)$, we see that errors $\|\wtC - \wtC'\|_{W^{p+1,\infty}(\Omega)}$ are controlled by measurement errors in $\|\bu^{(j)}-\bu^{'(j)}\|_{W^{p+3,\infty}(\Omega)}$ norm, so that the second term in the right hand side of \eqref{eq:stab3} implies a loss of two derivatives from the measurements to the quantity $\tau$. This explains Remark \ref{rem:stabtau}. 

\paragraph{Stability on $\tdiv C$.} The right-hand-side of \eqref{eq:divCrc} is a linear functional in the components of $C$, rational in the components of the measurements and its partial derivatives up to second order so that, as before, the stability on $\tdiv C$ is of the form
\begin{align}
    \|\tdiv C - \tdiv \wtC\|_{W^{p,\infty}} \le K' \sum_{j=1}^{6+N} \|\eps^{(j)}-\eps^{'(j)}\|_{W^{p+1,\infty}(\Omega)} + K'' \|C-C'\|_{W^{p,\infty}(\Omega)}, \qquad p\ge 0.
    \label{eq:stab4}
\end{align}
This ends the discussion on stability and the proof of Theorem \ref{thm:injstab}.

\section{Reconstructible tensors - Fulfilling Hypotheses \ref{hyp}.A-B} \label{sec:recons}

\subsection{The constant coefficient problem} \label{ssec:constant}

We now show that, in the same way that harmonic polynomials up to second order uniquely characterize a constant diffusion tensor in the scalar case, a constant elasticity tensor $C$ is uniquely characterized by elasticity solutions whose components are polynomials up to second order. 

\paragraph{Polynomial displacement fields.} Assume that $C$ is constant and let us denote by $\bc_i$ the $i$-th row (or column) of $\bc$ (i.e., $C$ in Voigt notation). 

It is straightforward to see that any displacement field $\bu$ with linear components is a solution to \eqref{eq:elasticity}. These solutions in fact allow us to construct a basis of strain tensors, i.e. such that $\eps^{(1)},\dots,\eps^{(6)}$ is the natural basis of $\Rm^6$ for every $\x\in X$. This can be achieved by considering the following solutions
\begin{align}
    \begin{split}
	\bu^{(1)}(\x) &= (x, 0, 0)^T, \quad \bu^{(2)}(\x) = (0,y,0)^T, \quad \bu^{(3)} (\x) = (0,0,z)^T, \\
	\bu^{(4)}(\x) &= \frac{1}{2}(0,z,y)^T, \quad \bu^{(5)}(\x) = \frac{1}{2}(z,0,x)^T, \quad \bu^{(6)} (\x) = \frac{1}{2}(y,x,0)^T.
    \end{split}
    \label{eq:linear}
\end{align}

Second, a displacement field with quadratic components, of the form 
\begin{align}
  \bu (\x) = \left( \frac{1}{2} \x\cdot P\x,\ \frac{1}{2} \x\cdot Q\x,\ \frac{1}{2} \x\cdot R\x  \right), \qquad P,Q,R \in S_3(\Rm),
  \label{eq:quadratic}
\end{align}
has a strain tensor, in Voigt notation, taking the form $\eps(\x) = x V_1 + y V_2 + z V_3$, where we have defined 
\begin{align*}
  V_1 &:= (P_{11},\ Q_{21},\ R_{31},\ Q_{31} + R_{21},\ P_{13} + R_{11},\ P_{12} + Q_{11})^T, \\
  V_2 &:= (P_{12},\ Q_{22},\ R_{32},\ Q_{32} + R_{22},\ P_{23} + R_{21},\ P_{22} + Q_{21})^T, \\
  V_3 &:= (P_{13},\ Q_{23},\ R_{33},\ Q_{33} + R_{23},\ P_{33} + R_{31},\ P_{32} + Q_{13})^T. 
\end{align*}
The correspondence $(P,Q,R)\mapsto (V_1,V_2,V_3)$ is bijective, with inverse ($V_{ij}$ denotes the $j$-th component of $V_i$)
\begin{align}
    \begin{split}
	P &= \left[
	    \begin{array}{ccc}
		V_{11} & V_{21} & V_{31} \\
		\cdot & V_{26}-V_{12} & \frac{1}{2}(V_{25}+V_{36}-V_{14}) \\
		\text{sym} & \cdot & V_{35}-V_{13}
	    \end{array}
	\right], \quad Q = \left[
	    \begin{array}{ccc}
		V_{16}-V_{21} & V_{12} & \frac{1}{2} (V_{36} + V_{14} -V_{25}) \\
		\cdot & V_{22} & V_{32} \\
		\text{sym} & \cdot & V_{34} - V_{23}
	    \end{array}
	\right], \\
	R &= \left[
	\begin{array}{ccc}
	    V_{15}-V_{31} &  \frac{1}{2} (V_{14} + V_{25} - V_{36}) & V_{13} \\
	    \cdot & V_{24}-V_{32} & V_{23} \\
	    \text{sym} & \cdot & V_{33}
	\end{array}
    \right],
    \end{split}
    \label{eq:Pinv}
\end{align}
so that prescribing one or the other is equivalent. Then the stress tensor is expressed as
\begin{align*}
  \sigma = \bc\ \eps = x\ \bc\ V_1 + y\ \bc\ V_2 + z\ \bc\ V_3.
\end{align*}
With this expression, $\bu$ is an elasticity solution if and only if the three scalar conditions below are fulfilled
\begin{align}
  \begin{split}
    \bc_1\cdot V_1 + \bc_6\cdot V_2 + \bc_5\cdot V_3 &= 0, \\
    \bc_6\cdot V_1 + \bc_2\cdot V_2 + \bc_4\cdot V_3 &= 0, \\
    \bc_5\cdot V_1 + \bc_4\cdot V_2 + \bc_3\cdot V_3 &= 0.    
  \end{split}
  \label{eq:conditions}  
\end{align}
This essentially leaves us with a 18-3 = 15-parameter family of quadratic solutions to the constant-coefficient elasticity system, and we now show that these solutions suffice to characterize $C$ uniquely.
Conditions \eqref{eq:conditions} can be written in the form \eqref{eq:orthogonality} where, defining $\{\bfe_i\}_{i=1}^6$ the natural basis of $\Rm^6$, the matrices $M^{1}, M^{2}, M^{3}$ take the form
\begin{align}
  \begin{split}
     M^1 &= V_1\otimes \bfe_1 + V_2\otimes \bfe_6 + V_3\otimes \bfe_5, \\
     M^2 &= V_1\otimes \bfe_6 + V_2\otimes \bfe_2 + V_3\otimes \bfe_4, \\
     M^3 &= V_1\otimes \bfe_5 + V_2\otimes \bfe_4 + V_3 \otimes \bfe_3.
  \end{split}
  \label{eq:Mmatrices} 
\end{align}
We will show in the next paragraph that, for various choices of $V_1,V_2,V_3$ satisfying \eqref{eq:conditions}, the corresponding matrices \eqref{eq:Mmatrices} span the hyperplane $\{\bc\}^\bot$ in $M_6(\Rm)$, thereby imposing enough orthogonality conditions on $\bc$ to determine it uniquely up to a multiplicative constant.

\paragraph{Rank maximality of quadratic displacement fields.} Since $\bc$ is invertible, let us denote $\bc_i^\star$ the $i$-th row (or column) of $\bc^{-1}$, so that $\bc_i\cdot \bc_j^\star = \delta_{ij}$. Since we can write $\bc$ as $\bc = \sum_{j=1}^6 \bc_j \otimes \bfe_j$, using the identity
\begin{align*}
  (U\otimes V):(S\otimes T) = (U\cdot S) (V\cdot T), \quad U,V,S,T\in \Rm^6,
\end{align*}
we can show that the space $\{\bc\}^\bot$, regarded as a hyperplane of $M_6(\Rm)$ of dimension 35, is spanned by the following family:
\begin{align}
  \{\bc\}^\bot = \text{span}\left\{ \bc_i^\star\otimes \bfe_j, \quad 1\le i,j\le 6, \quad i\ne j, \quad \bc_i^\star \otimes \bfe_i - \bc_{i+1}^\star\otimes \bfe_{i+1}, \quad 1\le i\le 5  \right\}.
  \label{eq:hyperplane}
\end{align}

Next, proceeding by exhaustion, we construct quadratic solutions giving rise to matrices $M^{1}, M^2, M^3$ spanning the family \eqref{eq:hyperplane}. 

\begin{itemize}
    \item We first consider quadratic displacement fields satisfying \eqref{eq:conditions}, and such that two vectors among $\{V_1,V_2,V_3\}$ vanish identically. 
	\begin{itemize}
	    \item If $V_1=V_2=0$, conditions \eqref{eq:conditions} read $V_3\bot \{\bc_3,\bc_4,\bc_5\}$, so that $V_3\in \text{span} \left\{ \bc_1^\star, \bc_2^\star, \bc_6^\star \right\}$, thus leading to matrices $M^1, M^2, M^3$ in the set $\{\bc_i^\star \otimes \bfe_j,\ i=3,4,5,\ j=1,2,6\}$.  
	    \item If $V_1=V_3=0$, conditions \eqref{eq:conditions} read $V_2 \bot \{\bc_2,\bc_4,\bc_6\}$, so that $V_2\in \text{span} \left\{ \bc_1^\star, \bc_3^\star, \bc_5^\star \right\}$, thus leading to matrices $M^{1},M^{2},M^{3}$ in the set $\{\bc_i^\star \otimes \bfe_j,\ i=1,3,5,\ j=2,4,6\}$.
	    \item If $V_2=V_3=0$, conditions \eqref{eq:conditions} read $V_1 \bot \{\bc_1,\bc_5,\bc_6\}$, so that $V_1\in \text{span} \left\{ \bc_2^\star, \bc_3^\star, \bc_4^\star \right\}$, thus leading to matrices $M^{1},M^{2},M^{3}$ in the set $\{\bc_i^\star \otimes \bfe_j,\ i=2,3,4,\ j=1,5,6\}$.
	\end{itemize}
    \item Secondly, we consider quadratic displacement fields satisfying \eqref{eq:conditions} and such that one vector among $\{V_1,V_2,V_3\}$ vanishes identically. 
	\begin{itemize}
	    \item If $V_1 = 0$, conditions \eqref{eq:conditions} read $\bc_6\cdot V_2 + \bc_5\cdot V_3 = \bc_2\cdot V_2 + \bc_4\cdot V_3 = \bc_4\cdot V_2 + \bc_3\cdot V_3 = 0$. One way to achieve this is by writing, for some free parameters $\alpha,\beta,\gamma$, 
		\begin{align*}
		    V_2 = \alpha \bc_6^\star + \beta \bc_2^\star + \gamma \bc_4^\star, \qquad V_3 = -\alpha \bc_5^\star - \beta \bc_4^\star - \gamma \bc_3^\star. 
		\end{align*}
		The matrices $M^1,M^2,M^3$ thus constructed take the form
		\begin{align*}
		    M^1 &= \alpha (\bc^\star_6\otimes \bfe_6 - \bc^\star_5\otimes \bfe_5) + \beta (\bc_2^\star \otimes \bfe_6 - \bc_4^\star \otimes \bfe_5) + \gamma (\bc_4^\star \otimes \bfe_6 - \bc_3^\star\otimes \bfe_5), \\
		    M^2 &= \alpha (\bc_6^\star\otimes \bfe_2 - \bc_5^\star\otimes \bfe_4) + \beta(\bc_2^\star\otimes \bfe_2 - \bc_4^\star\otimes \bfe_4) + \gamma(\bc_4^\star \otimes \bfe_2 - \bc_3^\star\otimes \bfe_4), \\
		    M^3 &= \alpha(\bc_6^\star \otimes \bfe_4 - \bc_5^\star\otimes \bfe_3) + \beta(\bc_2^\star \otimes \bfe_4 - \bc_4^\star\otimes \bfe_3) + \gamma(\bc_4^\star \otimes \bfe_4 - \bc_3^\star \otimes \bfe_3).      
		\end{align*}
	    \item If $V_2 = 0$, conditions \eqref{eq:conditions} read $\bc_1\cdot V_1 + \bc_5\cdot V_3 = \bc_6\cdot V_1 + \bc_4\cdot V_3 = \bc_5\cdot V_1 + \bc_3\cdot V_3 = 0$. One way to achieve this is by writing, for some free parameters $\alpha,\beta,\gamma$, 
		\begin{align*}
		    V_1 = \alpha \bc_1^\star + \beta \bc_6^\star + \gamma \bc_5^\star, \qquad V_3 = -\alpha \bc_5^\star - \beta \bc_4^\star - \gamma \bc_3^\star.
		\end{align*}
		The matrices $M^1,M^2,M^3$ thus constructed take the form
		\begin{align*}
		    M^1 &= \alpha (\bc^\star_1\otimes \bfe_1 - \bc^\star_5\otimes \bfe_5) + \beta (\bc_6^\star \otimes \bfe_1 - \bc_4^\star \otimes \bfe_5) + \gamma (\bc_5^\star \otimes \bfe_1 - \bc_3^\star\otimes \bfe_5), \\
		    M^2 &= \alpha (\bc_1^\star\otimes \bfe_6 - \bc_5^\star\otimes \bfe_4) + \beta(\bc_6^\star\otimes \bfe_6 - \bc_4^\star\otimes \bfe_4) + \gamma(\bc_5^\star \otimes \bfe_6 - \bc_3^\star\otimes \bfe_4), \\
		    M^3 &= \alpha(\bc_1^\star \otimes \bfe_5 - \bc_5^\star\otimes \bfe_3) + \beta(\bc_6^\star \otimes \bfe_5 - \bc_4^\star\otimes \bfe_3) + \gamma(\bc_5^\star \otimes \bfe_5 - \bc_3^\star \otimes \bfe_3). 
		\end{align*}
	\end{itemize}
\end{itemize}

Examining the five cases considered, we see that the set \eqref{eq:hyperplane} can be spanned by matrices $M^1,M^2,M^3$ generated by either of the above cases. This concludes the discussion. 

\begin{remark} The number of solutions required here is not sharp. In fact, it would be enough to construct a hyperplane (of dimension 20 instead of 35) of $S_6(\Rm)$ using well-chosen solutions, though it is not necessarily straightforward to find out which subfamily of \eqref{eq:hyperplane} of cardinality 20 spans the orthogonal space to $\{\bc\}$ in $S_6(\Rm)$.   
\end{remark}

\paragraph{The transversely isotropic case.}
We treat here a particular example of anisotropy, where the number of unknowns is reduced to $5$, and we show how the general method can be adapted. Assuming that the $\be_3$ (or ``$z$'') direction is the constant direction of isotropy, we decompose a transversely isotropic tensor, in Voigt notation, as 
\begin{align*}
    \bc &= \left[
    \begin{array}{cccccc}
	a & b & c & 0 & 0 & 0 \\
	b & a & c & 0 & 0 & 0 \\
	c & c & d & 0 & 0 & 0 \\
	0 & 0 & 0 & e & 0 & 0 \\
	0 & 0 & 0 & 0 & e & 0 \\
	0 & 0 & 0 & 0 & 0 & \frac{a-b}{2}
    \end{array} 
\right] 
\end{align*}

Locally, four well-chosen orthogonality constraints are enough to locate $\bc(\x)$ in the five-dimensional space describing it. As each displacement field gives rise to 3 orthogonality constrains, we expect that two well-chosen displacement fields $\bu^{(7)}, \bu^{(8)}$ in addition to the six ones forming a basis of strain tensors should suffice.

As in the fully anisotropic case, the family of linear displacement fields \eqref{eq:linear} forms a basis of strain tensors of elasticity solutions, and we now aim at finding two additional solutions in the form of well-chosen quadratic polynomial displacement fields. Since the unknowns are now the scalars $(a,b,c,d,e)$, we can rewrite the conditions \eqref{eq:conditions} for a quadratic displacement field to be an elasticity solution as orthogonality constraints on the vector $(a,b,c,d,e)$. In this set of variables, a displacement field of the form \eqref{eq:quadratic} solves the system of elasticity if and only if
\begin{align*}
    \left[
    \begin{array}{ccccc}
	V_{11} + \frac{1}{2} V_{26} & V_{12} - \frac{1}{2} V_{26} & V_{13} & 0 & V_{35} \\
	\frac{1}{2} V_{16} + V_{22} & -\frac{1}{2} V_{16} + V_{21} & V_{23} & 0 & V_{34} \\
	0 & 0 & V_{31}+V_{32} & V_{33} & V_{15} + V_{24}
    \end{array} \right] 
    \left[
    \begin{array}{c}
	a \\ b \\ c \\ d \\ e
    \end{array}
\right]
= \left[
\begin{array}{c}
    0 \\ 0 \\ 0
\end{array}
\right].
\end{align*}
The two additional quadratic solutions can be chosen as follows:
\begin{itemize}
    \item $\bu^{(7)}$ is constructed so that 
	\begin{align*}
	    \left[
		\begin{array}{ccccc}
		    V_{11} + \frac{1}{2} V_{26} & V_{12} - \frac{1}{2} V_{26} & V_{13} & 0 & V_{35} \\
		    \frac{1}{2} V_{16} + V_{22} & -\frac{1}{2} V_{16} + V_{21} & V_{23} & 0 & V_{34} \\
		    0 & 0 & V_{31}+V_{32} & V_{33} & V_{15} + V_{24}
	    \end{array} \right] = \left[ \begin{array}{ccccc}
		    -b & a & 0 & 0 & 0 \\
		    0 & -c & b & 0 & 0 \\
		    0 & 0 & -d & c & 0
	    \end{array} \right],
	\end{align*}
	all other coefficients $V_{ij}$ being set to zero. One possible solution of this is 
	\begin{align*}
	    (V_{12}, V_{11}, V_{23}, V_{33}, V_{21}, V_{31}) = (a,-b,b,c,-c,-d), 
	\end{align*}
	(all other coefficients set to zero) which, upon using \eqref{eq:Pinv}, yields the displacement field
	\begin{align*}
	    \bu^{(7)} = (-bx^2-2cxy-ay^2-2dxz, cx^2 + 2axy - bz^2, dx^2 + 2byz + cz^2).
	\end{align*}
    \item $\bu^{(8)}$ is constructed so that 	
	\begin{align*}
	    \left[
		\begin{array}{ccccc}
		    V_{11} + \frac{1}{2} V_{26} & V_{12} - \frac{1}{2} V_{26} & V_{13} & 0 & V_{35} \\
		    \frac{1}{2} V_{16} + V_{22} & -\frac{1}{2} V_{16} + V_{21} & V_{23} & 0 & V_{34} \\
		    0 & 0 & V_{31}+V_{32} & V_{33} & V_{15} + V_{24}
	    \end{array} \right] = \left[ \begin{array}{ccccc}
		    0 & 0 & 0 & 0 & 0 \\
		    0 & 0 & 0 & 0 & 0 \\
		    0 & 0 & 0 & -e & d
	    \end{array} \right],
	\end{align*}
	all other coefficients $V_{ij}$ being set to zero. One possible solution is
	\begin{align*}
	    (V_{33}, V_{15}) = (-e,d),
	\end{align*}
	(all other coefficients being set to zero) which, upon using \eqref{eq:Pinv}, yields the displacement field
	\begin{align*}
	    \bu^{(8)} = (0,0,dx^2 - ez^2). 
	\end{align*}
\end{itemize}
$\bu^{(7)}$ and $\bu^{(8)}$ were constructed to be solutions of the system of elasticity, and the orthogonality constraints they generate saturate a hyperplane of $\Rm^5$. Satisfying such orthogonality constraints, $(a,b,c,d,e)$ are uniquely determined up to a constant, which is in turn determined following the approach to reconstruct $\tau$ in the general case (see Sec. \ref{sec:rectau}).

\subsection{Proof of Theorem \ref{thm:recons}}

\subsubsection{Preliminaries: forward theory and interior regularity}\label{sec:prelim}

Here and below, if $\bu$ denotes a vector-valued function, we will loosely write $\bu\in H^s(X)$ to mean that $\bu$ has components in the Sobolev space $H^s(X)$. In order to prove Theorem \ref{thm:recons}, we will use the following theorem (see e.g. \cite[Chap. 6, Th. 1.11]{Marsden1983}): if $C$ is smooth, {\em uniformly pointwise stable} (as defined in Eq. \eqref{eq:UPS}) and $\partial X$ is smooth, then for every $f\in L^2(X)$ there exists a unique solution $\bfu\in H^2(X)$ to the problem 
\begin{align}
    \nabla\cdot (C\nabla\bfu) = \bff\quad (X), \quad \bfu|_{\partial X} = 0.
    \label{eq:elasticityf}
\end{align}
If $\bff\in H^s(X)$ then $\bfu\in H^{s+2}(X)$ for $s\ge 0$ with an estimate of the form
\begin{align}
    \|\bfu\|_{H^{s+2}(X)} \le C \|\bff\|_{H^{s}(X)}. 
    \label{eq:fwdellest}
\end{align}
Such estimates also hold in Sobolev norms $W^{s,p}$ for every $1<p<\infty$. Using a lift operator, if the boundary is smooth and problem \eqref{eq:elasticityf} is replaced by a boundary value problem with $\bff=0$ and $\bfu|_{\partial X} = \bg\in H^{s+\frac{3}{2}}(\partial X)$, one may obtain a similar inequality as \eqref{eq:fwdellest} upon replacing $\|\bff\|_{H^{s}(X)}$ by $\|\bg\|_{H^{s+\frac{3}{2}}(\partial X)}$ .

Additionally, we will use below that when the elasticity tensor $C$ is smooth and uniformly pointwise stable, interior regularity arguments as in \cite[Theorem 1 p309]{evans} (in the case of scalar elliptic PDEs) translate into the following: if $\bu$ solves \eqref{eq:elasticityf}, then for $V\subset\subset U\subset X$, there exists a constant $C_1(U,V,C)$ such that 
\begin{align}
    \|\eps\|_{H^1(V)} &\le C_1 (\|\bfu\|_{L^2(U)} + \|\bff\|_{L^2(U)}) \qquad \left(\eps_{ij} = \frac{1}{2} (\partial_i u_j + \partial_j u_i)\right).
    \label{eq:intreg}
\end{align}

Finally, we note that for any $\Omega\subset X$, the functionals $\F_1:[\C^{1}(\Omega,S_3(\Rm))]^{6}\to \C^{0}(\Omega)$ and $\F_2:[\C^1(\Omega,S_3(\Rm))]^{6+N} \to \C^0(\Omega)$, defined by 
    \begin{align}
	\{\eps^{(j)}\}_{1\le j\le 6} &\mapsto \F_1(\eps^{(1)},\dots,\eps^{(6)}) = \det_V (\eps^{(1)}, \dots, \eps^{(6)}), \label{eq:F1}\\
	\{\eps^{(j)}\}_{1\le j\le 6+N} &\mapsto \F_2(\eps^{(1)},\dots,\eps^{(6+N)}) = \sum_{M'\subset M, \# M' = 20} \N(M'):\N(M'), \label{eq:F2}
    \end{align}
    are continuous as polynomials of the entries of the strain tensors and their first partial derivatives.

\subsubsection{Near-constant elasticity tensors}

We now use the constructions from the constant coefficient problem in Sec. \ref{ssec:constant} as well as the continuity of displacement fields with respect to the elasticity coefficients in appropriate norms, to establish reconstructibility of near-constant tensors. 

\begin{proof}[Proof of Theorem \ref{thm:recons}(i)] For two displacement fields $\bu, \bu'$ solutions of \eqref{eq:elasticity} with respective elasticity tensors $C,C'$ (both uniformly pointwise stable and at least $\C^3$-smooth) and the same boundary condition $\bg\in H^{\frac{5}{2}}(\partial X)$, the difference $\bu-\bu'$ satisfies the following PDE
    \begin{align*}
	\nabla\cdot (C \nabla(\bu-\bu')) = \nabla\cdot ((C'-C)\bu') \qquad (X), \qquad (\bu-\bu')|_{\partial X} = 0,  
    \end{align*}
    so that, using \eqref{eq:fwdellest} with $s=2$, we obtain that 
    \begin{align*}
	\|\eps-\eps'\|_{H^3(X)} \le K \|C-C'\|_{\C^3(X)} \|\bu'\|_{H^3(X)} \le K' \|C-C'\|_{\C^3(X)} \|\bg\|_{H^{\frac{5}{2}}(\partial X)},
    \end{align*}
    for some constants $K,K'$. By Sobolev inequality $H^3(X)\to \C^{1,\frac{1}{2}}(\overline{X})\subset \C^1(\overline{X})$, we arrive at the estimate
    \begin{align}
	\|\eps-\eps'\|_{\C^1(\overline{X})} \le K \|C-C'\|_{\C^3(X)}. 
	\label{eq:Cest}
    \end{align}

    Now pick $C' = C_0$ a tensor with constant coefficients and construct solutions $\{\bu_0^{(j)}\}_{j=1}^{6+N}$ satisfying hypotheses \ref{hyp}, i.e. $\F_1 (\eps_0^{(1)}, \dots, \eps_0^{(6)})$ and $\F_2(\eps_0^{(1)}, \dots, \eps_0^{(6+N)})$ are bounded away from zero throughout $X$. We then see that combining the continuity of $\F_1, \F_2$ in \eqref{eq:F1}-\eqref{eq:F2} with estimate \eqref{eq:Cest} applied to $\eps^{(1)}, \dots, \eps^{(6+N)}$ shows that, if an elasticity tensor $C$ is close enough to $C_0$ in $\C^3(X)$-norm, then upon defining $\{\bu^{(j)}\}_{j=1}^{6+N}$ solutions of $\nabla\cdot(C\nabla\bu^{(j)}) = 0$ with boundary condition $\bu^{(j)}|_{\partial X} = \bu_0^{(j)}|_{\partial X}$, the difference $\|\eps^{(j)}-\eps_0^{(j)}\|_{\C^1(\overline{X})}$ should be so small that $\F_1 (\eps^{(1)}, \dots, \eps^{(6)})$ and $\F_2(\eps^{(1)}, \dots, \eps^{(6+N)})$ remain bounded away from zero throughout $X$, so that $C$ can be explicitely reconstructed from these strain tensors. Theorem \ref{thm:recons}(i) is proved.     
\end{proof}

\subsubsection{Runge approximation} \label{sec:runge}
 
We say that a differential operator $\L$ satisfies the Runge approximation property on $X$ if for every compact $\Omega\subset X$, every solution $v$ of $\L v = 0\; (\Omega)$ can be approximated in $L^2(\Omega)$ by solutions $u$ of $\L u = 0\; (X)$. Other approximation topologies can be used (e.g. $H^1 (\Omega)$ in \cite[Theorem 3.3]{Nakamura2005}), in fact here we will use a $\C^2$ approximation. The Runge approximation property is closely related to the unique continuation property, see e.g. \cite{Lax1956}.

Unlike for scalar elliptic PDE's, the literature on unique continuation results for elliptic systems is sparse. In the context of elasticity, it has been proved for the Lam\'e system (and more generally for elliptic systems with an iterated Laplacian as diagonal part) using doubling inequalities in \cite{Alessandrini2008}, and for the transversely isotropic case in \cite{Nakamura2005}.

When the Runge approximation is available, following ideas from the first and third author presented in \cite{Bal2012} and generalized in other contexts in e.g. \cite{Monard2012a,Bal2013}, one becomes able to fulfill the hypotheses of reconstructibility \ref{hyp} by constructing local solutions satisfying appropriate qualitative behavior. This is what we now do in the proof of Theorem \ref{thm:recons}(ii).

\begin{proof}[Proof of Theorem \ref{thm:recons}(ii)] We decompose the proof in three steps.  
    \paragraph{Step 1. Local solutions with constant coefficients:} 
    Fix $\x_0\in X$ and $B_{3r} \equiv B_{3r}(\x_0)\subset X$ a ball of radius $3r$ ($r$ tuned hereafter) centered at $\x_0$, and denote $C_0 := C(\x_0)$. Following the approach in section \ref{ssec:constant}, we first construct solutions to the problem with frozen constant coefficients, i.e. such solutions, which we denote $\{\bfu_0^{(j)}\}_{j=1}^{6+N}$ (each with strain tensor $\eps_0^{(j)}$), solve 
    \begin{align}
	\nabla\cdot(C_0 \eps_0^{(j)}) = 0\quad (\Rm^3), \qquad 1\le j\le 6+N.
	\label{eq:frozensol}
    \end{align}
    This family is constructed so that Hypotheses \ref{hyp}.A-B are satisfied throughout $B_{3r}$ (in fact, they fulfill these hypotheses globally). 
    
    \paragraph{Step 2. Local solutions with varying coefficients:}
    From solutions $\{\bfu_{0}^{(j)}\}_{j=1}^{6+N}$, we construct a second family of solutions $\{\bfu_{r}^{(j)}\}_{j=1}^{6+N}$ (each with strain tensor $\eps_{r}^{(j)}$) via the following equation
    \begin{align}
	\nabla\cdot(C \eps_r^{(j)}) = 0\quad (B_{3r}), \qquad \bfu_r^{(j)} |_{\partial B_{3r}} = \bfu_0^{(j)}, \qquad 1\le j\le 6+N.
	\label{eq:localsol}
    \end{align}
    
    The difference of both solutions satisfies, for $1\le j\le 6+N$,
    \begin{align}
	\nabla\cdot(C (\eps_r^{(j)} - \eps_0^{(j)}) ) = \nabla\cdot ((C_0-C)\eps_0^{(j)}) \quad (B_{3r}), \quad (\bfu_r^{(j)}-\bfu_0^{(j)})|_{\partial B_{3r}} = 0,
	\label{eq:localdiff}
    \end{align}
    where the right-hand side is smooth since $C$ is assumed to be smooth and so is $\eps_0^{(j)}$. Then using estimate \eqref{eq:fwdellest} in high enough Sobolev norms combined with Sobolev inequalities \cite[Th. 6 p270]{evans}, we arrive at estimates of the form 
    \begin{align*}
	\|\bfu_r^{(j)} - \bfu_0^{(j)} \|_{\C^{2,\frac{1}{2}}(B_{3r})} \le C \|\bfu_r^{(j)} - \bfu_0^{(j)} \|_{H^4(B_{3r})} \le C' \|\nabla\cdot ((C_0-C)\eps_0^{(j)})\|_{H^2(B_{3r})},
    \end{align*}
    so that
    \begin{align}
	\lim_{r\to 0} \quad  \max_{1\le j\le 6+N} \|\bfu_r^{(j)} - \bfu_0^{(j)} \|_{\C^2(B_{3r})} = 0.
	\label{eq:localestimate}
    \end{align}

    \paragraph{Step 3. Runge approximation (control from the boundary $\partial X$):} Assume $r$ has been fixed at this stage. By virtue of the Runge approximation property, for every $\delta >0$ and $1\le j\le 6+N$, there exists $\bfg_\delta^{(j)}\in H^{\frac{1}{2}}(\partial X)$ such that 
    \begin{align}	
	\|\bfu_\delta^{(j)} - \bfu_r^{(j)}\|_{L^2(B_{3r})} \le \delta, \where\; \bfu_\delta^{(j)} \text{ solves } \eqref{eq:elasticity} \text{ with } \bfu_\delta^{(j)}|_{\partial X} = \bfg_\delta^{(j)}.
	\label{eq:L2est}
    \end{align}
    Since $C$ is smooth and $\nabla\cdot(C(\eps_\delta^{(j)} - \eps_r^{(j)}))=0$ thoughout $B_{3r}$, estimate \eqref{eq:intreg} with $\bff=0$ implies that
    \begin{align*}
	\|\eps^{(j)}_\delta-\eps^{(j)}_r\|_{H^1(B_{2r})} \le C_1(r) \|\bfu_\delta^{(j)} - \bfu_{r}^{(j)}\|_{L^2(B_{3r})}.
    \end{align*}
    Moreover, using that $\nabla\cdot(C \partial_{pq} (\eps_\delta^{(j)} - \eps_r^{(j)})) = - \nabla\cdot( (\partial_{pq}C) (\eps_\delta^{(j)} - \eps_r^{(j)}))$ for every $1\le p\le q\le 3$, and assuming that $C$ is at least of class $\C^3$, we can find a constant $C_2(r,C)$ (we drop the dependency on $C$ below) such that 
    \begin{align*}
	\|\eps^{(j)}_\delta-\eps^{(j)}_r\|_{H^3(B_{r})} \le C_2(r,C) \|\eps^{(j)}_\delta-\eps^{(j)}_r\|_{H^1(B_{2r})} \le C_3(r) \|\bfu_\delta^{(j)} - \bfu_{r}^{(j)}\|_{L^2(B_{3r})}\le C_3(r)\delta.
    \end{align*}
    By Sobolev inequality $H^3\to \C^{1,\frac{1}{2}}\subset \C^1$, we deduce that 
    \begin{align*}
	\|\eps^{(j)}_\delta-\eps^{(j)}_r\|_{\C^1(B_r)} \le C_4(r)\delta, \qquad 1\le j\le 6+N.
    \end{align*}
    Since $r$ is fixed at this stage, we deduce that 
    \begin{align}
	\lim_{\delta\to 0} \quad \max_{1\le j\le 6+N} \|\eps_\delta^{(j)} - \eps_r^{(j)} \|_{\C^1(B_r)} = 0.
	\label{eq:epsilonestimate}
    \end{align}
    
    \paragraph{Completion of the argument:} Theorem \ref{thm:injstab} states that $C$ is reconstructible if both functionals $\F_1,\F_2$ are bounded away from zero for some properly chosen solutions. Fixing again $\x_0\in X$ and $B_r$, as above, Step 1 established that $\F_1(\eps_0^{(1)}, \dots, \eps_0^{(6)})$ and $\F_2(\eps_0^{(1)}, \dots, \eps_0^{(6+N)})$ were bounded away from zero over $B_r$, where the $\eps_0^{(j)}$'s were strain tensors associated with a problem with frozen coefficients. Due to limits \eqref{eq:localestimate} and \eqref{eq:epsilonestimate}, there exists a small $r>0$, then a small $\delta>0$ and local solutions $(\bu_\delta^{(1)}, \dots, \bu_\delta^{(6+N)})$ such that $\max_{1\le j\le 6+N} \|\eps^{(j)}_\delta-\eps_0^{(j)}\|_{\C^1(B_r(x_0))}$ is so small that, by continuity of $\F_1$ and $\F_2$ mentioned above, $\F_1(\eps_\delta^{(1)}, \dots, \eps_\delta^{(6)})$ and $\F_2(\eps_\delta^{(1)}, \dots, \eps_\delta^{(6+N)})$ remain uniformly bounded away from zero over $B_r$. Hypotheses \ref{hyp} are thus satisfied over $B_r$ by the family $\{\bu_\delta^{(j)}\}_{j=1}^{6+N}$ which is controlled by boundary conditions. $C$ is thus reconstructible over $B_r$ with a local stability as in Theorem \ref{thm:injstab}. The proof of Theorem \ref{thm:recons}(ii) is complete.
\end{proof}

\subsection*{Acknowledgements} GB is partially supported by NSF grant DMS-1408867, FM by NSF grant DMS-1025372, while GU acknowledges support from NSF. FM would like to thank C\'edric Bellis for pointing out reference \cite{Mehrabadi1990}.

\bibliographystyle{siam}

\end{document}